\documentclass[review,onefignum,onetabnum]{siamart190516}

\usepackage{lipsum}
\usepackage{amsfonts}
\usepackage{amsmath}
\usepackage{bm}
\usepackage{array}
\usepackage{graphicx}
\usepackage{epstopdf}
\usepackage{algorithmic}
\usepackage{longtable}
\ifpdf
  \DeclareGraphicsExtensions{.eps,.pdf,.png,.jpg}
\else
  \DeclareGraphicsExtensions{.eps}
\fi

\usepackage{graphicx,epstopdf} 
\usepackage[caption=false]{subfig} 

\newsiamremark{remark}{Remark}
\newsiamremark{hypothesis}{Hypothesis}
\crefname{hypothesis}{Hypothesis}{Hypotheses}
\newsiamthm{claim}{Claim}

\headers{smoothing algorithms over the stiefel manifold}{J. Zhu, J. Huang, L. Yang, Q. Li}

\title{Smoothing algorithms for nonsmooth and nonconvex minimization over the stiefel manifold\thanks{2022/10/17.
\funding{The work of the third author was supported in part by the National Natural Science Foundation of China under grant 12171488, and by the Guangdong Province Key Laboratory of Computational National Science at the Sun Yat-sen University, grant 2020B1212060032. The work of the last author was supported in part by the National Natural Science Foundation of China under grant 11971499, and by the Guanddong Province Key Laboratory of Computational National Science at Sun Yat-sen University, grant 2020B1212060032.}}}

\author{Jinlai Zhu\thanks{School of Mathematics, Sun Yat-sen University, Guangzhou 510275, China (\email{zhujlai@mail2.sysu.edu.cn}).}
	\and Jianfeng Huang\thanks{School of Financial Mathematics and Statistics, Guangdong University of Finance, Guangzhou 510521, China (\email{47-072@gduf.edu.cn}).}
	\and Lihua Yang\thanks{School of Mathematics, Guangdong Provincial Key Laboratory of Computational Science, Sun Yat-sen University, Guangzhou 510275, China (\email{mcsylh@mail.sysu.edu.cn}).}
	\and Qia Li\thanks{Corresponding author. School of Computer Science and Engineering, Guangdong Province Key Laboratory of Computational Science, Sun Yat-sen University, Guangzhou 510275, China (\email{liqia@mail.sysu.edu.cn}).}}
\usepackage{amsopn}

\ifpdf
\hypersetup{
  pdftitle={Smoothing algorithms for nonsmooth and nonconvex minimization over the stiefel manifold},
  pdfauthor={J. Zhu, J. Huang, L. Yang, Q. Li}
}
\fi

\begin{document}

\maketitle

\begin{abstract}
We consider a class of nonsmooth and nonconvex optimization problems over the Stiefel manifold where the objective function is the summation of a nonconvex smooth function and a nonsmooth Lipschitz continuous convex function composed with an linear mapping. We propose three numerical algorithms for solving this problem, by combining smoothing methods and some existing algorithms for smooth optimization over the Stiefel manifold. In particular, we approximate the aforementioned nonsmooth convex function by its Moreau envelope in our smoothing methods, and prove that the Moreau envelope has many favorable properties. Thanks to this and the scheme for updating the smoothing parameter, we show that any accumulation point of the solution sequence generated by the proposed algorithms is a stationary point of the original optimization problem. Numerical experiments on building	graph Fourier basis are conducted to demonstrate the efficiency of the proposed algorithms.
\end{abstract}


\section{Introduction}
In this paper, we consider minimizing a nonsmooth and nonconvex objective function with orthogonality constraints:
\begin{equation}\label{ep1}
\begin{cases}
		\min\limits_{X\in\mathbb{R}^{n\times p}} F(X):=f(BX)+h(X)\\
		{\rm s.t.}\ \ X^\top X=I_{p},\end{cases}
\end{equation}
where $p\leq n,$ $B\in\mathbb{R}^{q\times n},$ $f:\mathbb{R}^{q\times p}\to\mathbb{R}$ is convex, nonsmooth and Lipschitz continuous with Lipschitz constant $L_f,$  $h:\mathbb{R}^{n\times p}\to\mathbb{R}$ is smooth, possibly nonconvex and $\nabla h$ is Lipschitz continuous with Lipschitz constant $L_h.$ In particular, the feasible region of problem $\cref{ep1}$ is called the Stiefel manifold, denoted by
\begin{displaymath}
	\mathcal{S}_{n,p}:=\{X\in\mathbb{R}^{n\times p}: X^\top X=I_{p}\}.
\end{displaymath}\par
This model has a wide range of applications, such as the $\ell_1$ sparse principal component analysis (PCA) \cite{cai2013sparse,ma2013sparse,xiao2021exact}, and the $\ell_{2,1}$-norm regularized PCA \cite{cai2013sparse,ma2013sparse,ulfarsson2008sparse}, the physical compression model \cite{ozolicnvs2013compressed} and the feature selection problem \cite{tang2012unsupervised,yang2011l2}. In this paper we are particularly interested in its application to the graph Fourier basis problem, which is described in the following subsection. For more examples of nonsmooth optimization over the Stiefel manifold, we refer the reader to \cite{absil2019collection}.
\subsection{Application to graph Fourier basis}
\label{sec:GFB}
Next we focus on the application of problem $\cref{ep1}$ to the graph Fourier basis problem (GFB) that arises in graph signal processing. Graph signal processing is an emerging research field \cite{sandryhaila2014discrete,sardellitti2017graph,shuman2013emerging,shuman2016vertex}, which extends the theory and methods of classical signal processing and has been successfully applied in many practical problems. In particular, graph Fourier transform (GFT) is a fundamental tool of graph signal analysis \cite{sandryhaila2013discrete}. Alternative definitions of GFT have been introduced in \cite{pesenson2008sampling,sandryhaila2014discrete,sardellitti2017graph,shuman2013emerging,singh2016graph,zhu2012approximating} based on different motivations, such as establishing a basis with minimal variation, filtering signals defined on graphs, approximation of signal supported on graphs, and so on. Below, we briefly review three different definitions of GFT. The first one is based on spectral graph theory, where the graph Laplacian plays a central role, see e.g., \cite{shuman2013emerging} and the references therein. This definition works only for the undirected graphs, where the Laplacian eigenvectors minimize the $\ell_2$ norm variation of a graph signal. The second definition is based on the Jordan decomposition of the adjacency matrix and defines the associated generalized eigenvectors as the GFT basis \cite{sandryhaila2013discrete,sandryhaila2014discrete}, which is suitable for both undirected and directed graphs. Since the numerical computation of the Jordan decomposition often leads to well-known numerical instability, the authors in \cite{sardellitti2017graph} discussed the general situation of directed graphs, and introduced the third way to construct the Fourier basis of the graph. To illustrate this definition, we introduce some notations for graph signals. Let $\mathcal{G}=(\mathcal{V},\mathcal{E})$ be a connected and weighted graph consisting of a set of nodes (or vertices) $\mathcal{V}=\{1,2,\cdots,N\}$ and a set of edges $\mathcal{E}=\{w_{ij}\}_{i,j\in\mathcal{V}}$ such that $w_{ij}>0$ if there is a link from node $i$ to node $j$ and $w_{ij}=0$ otherwise. The matrix $W=[w_{ij}]\in\mathbb{R}^{N\times N}$ is called the adjacency matrix. A graph signal $s$ is a real-valued function defined on $\mathcal{V},$ and can be regarded as a vector in $\mathbb{R}^{N}.$ Let $Z=[Z_{1},\cdots,Z_{N}]\in\mathbb{R}^{N\times N}$ be the Fourier basis of the graph $\mathcal{G}.$ Then, as in \cite{sardellitti2017graph}, the first basis vector $Z_{1}$ is defined as $\mathbf{1}/\sqrt{N}$ where $\mathbf{1}\in\mathbb{R}^{N}$ is an all-one vector, and the remaining vectors $\tilde{Z}=[Z_{2},\cdots,Z_{N}]$ are constructed by solving the following optimization problem:
\begin{equation}\label{gfb1}
	\begin{cases}
		\min\limits_{\tilde{Z}\in\mathbb{R}^{N\times (N-1)}} \sum\limits_{m=2}^N\sum\limits_{i,j=1}^{N}\omega_{ij}[(\tilde{Z}_{m})_j-(\tilde{Z}_{m})_i]_+\\
		{\rm s.t.}\ \ \tilde{Z}^\top \tilde{Z}=I_{N-1}, \tilde{Z}^\top Z_{1}=0_{N-1},\end{cases}
\end{equation}
 where $[y]_{+}:=\max\{y,0\},$ for $y\in\mathbb{R}.$ In fact, we can see that problem \cref{gfb1} is a special case of problem $\cref{ep1}$ by suitable reformulation. Let $|\mathcal{E}|$ be the cardinality of $\mathcal{E},$ and suppose the $k$-th edge is from node $i_k$ to node $j_k$ for $k=1,\cdots,|\mathcal{E}|.$ Define $\tilde{B}\in\mathbb{R}^{|\mathcal{E}|\times N}$ as follows,
\begin{equation}\label{etb}
	(\tilde{B})_{k\ell}=\begin{cases}
		1,&\mbox{if}\ \ell=j_k,\\
		-1,&\mbox{if}\ \ell=i_k,\\
	    0,&\mbox{otherwise}.\end{cases}
\end{equation}
It is easy to check that the objective function of problem \cref{gfb1} can be rewritten as $\Psi(\tilde{B}\tilde{Z}),$ where $\Psi:\mathbb{R}^{|\mathcal{E}|\times (N-1)}\to\mathbb{R}$ is defined at $Y\in\mathbb{R}^{|\mathcal{E}|\times (N-1)}$ as follows \begin{equation}\label{ep}
\Psi(Y)=\sum\limits_{k=1}^{|\mathcal{E}|}w_{i_k,j_k}\sum\limits_{j=1}^{N-1}[Y_{k,j}]_+.
\end{equation}
Furthermore, let $\tilde{V}\in\mathbb{R}^{N\times(N-1)}$ be an orthogonal basis of the subspace $\{z\in\mathbb{R}^{N\times 1}: z^\top Z_{1}=0\},$ i.e., $\tilde{V}^\top \tilde{V}=I_{N-1}$ and $\tilde{V}^\top Z_{1}=0_{N-1}.$ Then we immediately have that
\begin{equation}\label{eb}
\begin{split}
&\quad\{\tilde{Z}\in\mathbb{R}^{N\times (N-1)}: \tilde{Z}^\top\tilde{Z}=I_{N-1}, \tilde{Z}^\top Z_{1}=0_{N-1}\}	\\
&=\{\tilde{V}X: X^\top X=I_{N-1}, X\in\mathbb{R}^{(N-1)\times(N-1)}\}.
\end{split}
\end{equation}
By \cref{ep} and \cref{eb}, $\tilde{Z}^{\ast}$ is an optimal solution of problem \cref{gfb1} if and only if $\tilde{Z}^{\ast}=\tilde{V}X^{\ast}$ and $X^{\ast}$ is an optimal solution of the following problem 
\begin{equation}\label{ep2}
	\begin{cases}
		\min\limits_{X\in\mathbb{R}^{(N-1)\times (N-1)}} \Psi(\tilde{B}\tilde{V}X)\\
		{\rm s.t.}\ \ X^\top X=I_{N-1}.\end{cases}
\end{equation}
Obviously, problem $\cref{ep1}$ reduces to problem $\cref{ep2}$ when $f=\Psi, h=0$ and $B=\tilde{B}\tilde{V}.$
\subsection{Existing methods}
\label{sec:existing} 
There are many efficient algorithms for smooth optimization over the Stiefel manifold, such as gradient-based algorithms \cite{abrudan2008steepest, boumal2019global, manton2002optimization, nishimori2005learning}, conjugate gradient algorithms
\cite{abrudan2009conjugate, edelman1998geometry}, geodesic methods \cite{abrudan2008steepest,absil2009optimization,edelman1998geometry}, projection methods \cite{absil2009optimization,absil2012projection,manton2002optimization}, nonretraction algorithms \cite{gao2018new,gao2019parallelizable,wang2020multipliers,xiao2020class,xiao2021exact}. However, these methods are not suitable for the nonsmooth case. In this subsection, we briefly review three categories of algorithms for nonsmooth optimization over the Stiefel manifold, namely subgradient, splitting and alternating, and proximal point methods.\par
\textbf{Subgradient methods.}
The convergence of the Riemannian subgradient type method are studied in \cite{bento2017iteration,ferreira1998subgradient,ferreira2019iteration,zhang2016first} when the objective function is geodesically convex on the Riemannian manifold. Due to the availability of a geodesic version of the convex subgradient inequality, the general analysis of convex optimization in Euclidean space can be generalized to geodesic convex optimization on Riemannian manifolds. In particular, an asymptotic convergence result was established for the first time in \cite{ferreira1998subgradient}, while \cite{bento2017iteration,ferreira2019iteration} give the global rate of convergence, for the Riemannian subgradient method. In \cite{zhang2016first}, Zhang and Sra consider the case where the objective function is strongly geodesically convex on the Riemannian manifold and show that the rate of the Riemannian projected subgradient method can be improved. However, these results do not apply to problem $\cref{ep1}.$ This is because every continuous function that is geodesically convex over a compact Riemannian manifold is a constant \cite{bishop1969manifolds,yau1974non}. In \cite{hosseini2018line}, based on the descent direction of $\varepsilon$-subgradient and the Wolfe condition, Hosseini et al. propose a nonsmooth Riemannian manifold line search algorithm and expand the classical BFGS algorithm to nonsmooth functions over Riemannian manifolds. In \cite{hosseini2017riemannian}, Hosseini and Uschmajew present a Riemann gradient sampling algorithm. Roughly speaking, all the methods studied in \cite{hosseini2018line,hosseini2017riemannian} make use of subgradient information to construct a quadratic program to find a descent direction. Since the number of subgradient required is usually larger than the dimension of $\mathcal{S}_{n,p},$ the quadratic program can be difficult to solve for high dimensional problems on the Stiefel manifold. In \cite{li2021weakly}, Li et al. propose several Riemannian subgradient type methods, namely Riemannian subgradient, incremental subgradient and stochastic subgradient methods for minimizing weakly convex functions on Stiefel manifolds. They establish the Riemannian subgradient inequality on the Stiefel manifold, and then use this inequality to estimate the iteration complexity and local convergence rate of these methods.\par
\textbf{Splitting and alternating methods.} 
Clearly, the main difficulty in solving problem $\cref{ep1}$ is that the objective function has a nonsmooth term and the constraint is nonconvex. If there is only one of them, the problem is relatively easier to solve. Hence, a natural choice for solving problem $\cref{ep1}$ is splitting and using the alternating direction method of multipliers (ADMM) \cite{boyd2011distributed} or augmented Lagrangian method (ALM). In \cite{lai2014splitting}, Lai and Osher propose the split orthogonal constraint method (SOC), which first introduces an auxiliary variable to split the objective function and orthogonal constraint, and then applies ADMM solving the equivalent split problem. In \cite{kovnatsky2016madmm}, Kovnatsky et al. propose the manifold alternating direction method of multipliers (MADMM), which is an extension of the classic ADMM method for nonsmooth optimization problem over manifold constraints. However, to the best of our knowledge, the convergence guarantees for SOC and MADMM are still lacking in the literature. Although some recent works have analyzed the convergence of ADMM for nonconvex problems \cite{wang2019global,zhang2020primal}, their results require further assumptions which are not satisfied by problem $\cref{ep1}.$ Recently, in \cite{chen2016augmented}, Chen et al. present the proximal alternating minimized augmented Lagrangian method (PAMAL) based on the ALM. Unlike SOC, PAMAL applies the ALM framework to splitting problem and the resulting subproblems are solved by block coordinate descent (BCD) method. As shown in \cite{chen2020proximal}, the numerical performance of the PAMAL is very sensitive to the choice of penalty and other algorithm parameters.\par
\textbf{Proximal point methods.} 
In \cite{chen2020proximal}, Chen et al. propose the manifold proximal gradient method (ManPG) and its speed-up version ManPG-Ada. At each iteration, it computes the proximal mapping restricted to the tangent space at the current iterate of the Stiefel manifold, which is a nonsmooth convex optimization problem. As shown in \cite{chen2020proximal}, both ManPG and ManPG-Ada perform numerically better than SOC and PAMAL methods. Recently, Huang and Wei develop an accelerated algorithm (AManPG) by incorporating some extrapolation techniques into ManPG in \cite{huang2019extension}. In \cite{huang2022riemannian}, they also propose the modified Riemannian proximal gradient method (RPG) and its accelerated version (ARPG), of which the proximal mapping subproblems are nonsmooth and nonconvex. It is worth noting that all the methods above require to solve a nonsmooth optimization problem at each iteration that has no closed-form expression. Hence, the efficiency of these methods are highly affected by the computational cost of solving the involved subproblems. In particular, the subproblems in RPG and ARPG are usually of higher computational cost then those of ManPG due to their nonsmoothness and nonconvexity. As shown in \cite{huang2022riemannian}, RPG and ARPG are generally slower than AManPG and ManPG-Ada for nonsmooth optimization over the Stiefel manifold.
\subsection{Our contributions}
\label{sec:our}
In this paper, we develop numerical algorithms for solving problem $\cref{ep1}$ by using the technique of smoothing approximation methods. Smoothing approximation methods, which approximate the original nonsmooth problem by a sequence of smooth problems, has been studied for decades and developed rich theories and powerful algorithms, see \cite{bian2020smoothing,chen1996class,chen2012smoothing,fukushima1998globally,hintermuller2013nonconvex,kreimer1992nondifferentiable,nesterov2005smooth,nesterov2007smoothing} for example. Specially, in \cite{bian2020smoothing}, the authors formally present several important requirements on smooth approximation functions which are crucial for the convergence analysis of corresponding smoothing algorithms. In order to solve problem $\cref{ep1},$ we use the Moreau envelope \cite{bauschke2011convex,moreau1965proximite,rockafellar2009variational} of the function $f$ as the smooth approximation function, and prove that for a general convex and Lipschitz continuous function, its Moreau envelope satisfies all the requirements on a smoothing function mentioned in \cite{bian2020smoothing}. To the best of our knowledge, our work is the first to establish the function value and gradient consistency properties of the Moreau envelope, which play important roles in the convergence analysis of smooth approximation methods. Then we use three different existing algorithms, including gradient projection/reflection with correction (GPC/GRC) algorithms \cite{gao2018new,wang2020multipliers} and the Riemannian gradient descent algorithm (RGD) \cite{absil2009optimization} to tackle the resulting smooth optimization problem respectively. To speed up the convergence, we also incorporate a monotone line-search scheme into GPC and GRC. Moreover, the smoothing parameter is properly updated in our smoothing methods to guarantee convergence of the proposed algorithms to a stationary point of problem $\cref{ep1}.$ Compared with the aforementioned three categories of algorithms, the proposed algorithms are more suitable for solving problem $\cref{ep1},$ especially when the proximal mapping of $f$ rather than $f\circ B$ has closed form expression, e.g., GFB problem. First, despite their simplicity, subgradient methods usually need a good initial guess which is not easy to obtain in practice. Second, splitting and alternating methods have no rigorous convergence guarantees when applied to problem $\cref{ep1}.$ Third, in this case ManPG type proximal point methods have to use other nonsmooth convex optimization solver instead of the semi-smooth Newton method for solving the subproblems involved at each iteration, which may reduce the computational efficiency. By contrast, the proposed algorithms do not need to solve expensive subproblems and thus have relatively low computational complexity at each iteration. Besides, they are guaranteed to globally converge to a stationary point of problem $\cref{ep1}.$
\subsection{Notations}
\label{sec:notation}
Let $\langle X,Y\rangle={\rm tr}(X^{T}Y)$ denote the Euclidean inner product of two matrices $X,Y\in\mathbb{R}^{n\times p},$ where ${\rm tr}(A)$ represents the trace of matrix $A\in\mathbb{R}^{p\times p}.$ We use $\|\cdot\|_F$ and $\|\cdot\|_2$ to denote the Frobenius norm and 2-norm respectively. Let $\mathcal{B}(C,r)=\{X\in\mathbb{R}^{n\times p}:\|X-C\|_F\leq r\}$ denote the ball with center $C$ and radius $r.$ For a smooth function $\phi,$ $\nabla \phi$ represents its Euclidean gradient.\par
\subsection{Organization}
\label{sec:organization}
The rest of this article is organized as follows. In \Cref{sec:new}, we introduce some preliminaries on manifold optimization. We use the Moreau envelope as the smooth approximation function and develop three smoothing algorithms for solving problem $\cref{ep1}$ in \Cref{sec:nu}. In \Cref{sec:red}, we establish the grobal convergence of the proposed algorithms. Numerical results of the proposed algorithms on solving GFB problems are reported in \Cref{sec:num}. Finally, we conclude this paper in \Cref{sec:clu}.
\section{Preliminaries on manifold optimization}
\label{sec:new}
In this section, we introduce some preliminaries on manifold optimization. It is well known that the tangent space of a point $X\in\mathcal{S}_{n,p}$ is $\{Y\in\mathbb{R}^{n\times p}: Y^\top X+X^\top Y=0_{p\times p}\},$ which is denoted by $\mathcal{T}_{X}\mathcal{S}_{n,p}.$ According to \cite{absil2009optimization}, the projection of $Y$ onto $\mathcal{T}_{X}\mathcal{S}_{n,p}$ is given by
\begin{equation}\label{proj}
{\rm Proj}_{\mathcal{T}_{X}\mathcal{S}_{n,p}}(Y)=(I_n-XX^\top)Y+\frac{1}{2}X(X^\top Y-Y^\top X).
\end{equation}
A function $\varphi$ is said to be locally Lipschitz continuous in $\mathbb{R}^{n\times p}$ if for any $X\in\mathbb{R}^{n\times p}$ it is Lipschitz continuous within the neighborhood of $X.$ Note that if $\varphi$ is locally Lipschitz continuous in $\mathbb{R}^{n\times p},$ then it is also locally Lipschitz continuous when restricted to the embedding submanifold $\mathcal{S}_{n,p}.$ The Clark subdifferential of a nonsmooth function is defined as follows:
\begin{definition}(${\rm generalized\ Clarke\ subdifferential}$ \cite{clarke1990optimization,rockafellar2009variational})
For a locally Lipschitz continuous function $\varphi$ on $\mathbb{R}^{n\times p},$ the generalized directional derivative of $\varphi$ with respect to $X\in\mathbb{R}^{n\times p}$ in the direction $D$ is defined as
\begin{displaymath}
	\varphi^{\circ}(X,D):=\lim\limits_{Y\to X}\sup\limits_{t\downarrow 0}\frac{\varphi(Y+tD)-\varphi(Y)}{t}.
\end{displaymath}	
The Clark subdifferential (subdifferential for short) of $\varphi$ at $X\in\mathbb{R}^{n\times p}$ is defined by
\begin{displaymath}
\partial\varphi(X):=\{\varsigma\in\mathbb{R}^{n\times p}: \langle \varsigma,D\rangle\leq \varphi^{\circ}(X,D), \forall D\in\mathbb{R}^{n\times p}\}.
\end{displaymath}	
\end{definition}\par
The set $\partial\varphi(X)$ is closed and convex. If $\varphi$ is continuously differentiable at $X,$ then $\partial\varphi(X)=\{\nabla\varphi(X)\}.$ Furthermore, if $\varphi$ is convex, then $\partial\varphi(X)$ reduces to the traditional subdifferential in convex analysis, that is,
\begin{displaymath}
\partial\varphi(X):=\{Y\in\mathbb{R}^{n\times p}: \varphi(Z)-\varphi(X)-\langle Y,Z-X\rangle\geq 0, \forall Z\in\mathbb{R}^{n\times p}\}.
\end{displaymath}
In particular, for the objective function $F$ in problem $\cref{ep1},$ we have that
\begin{equation}\label{epf}
\partial F(X)=\partial(f\circ B)(X)+\nabla h(X)=B^\top\partial f(BX)+\nabla h(X).
\end{equation}
According to \cite[Theorem 4.1 and 5.1]{yang2014optimality}, we can characterize the first-order optimality of problem $\cref{ep1}$ and define its first-order stationary point as follows.
\begin{definition}\label{df1}
A point $X\in\mathcal{S}_{n,p}$ is called a first-order stationary point of problem $\cref{ep1}$ if it satisfies the first-order optimality condition $0\in{\rm Proj}_{\mathcal{T}_{X}\mathcal{S}_{n,p}}(\partial F(X)).$
\end{definition}\par
Let $X\in\mathcal{S}_{n,p},$ by \cref{proj}, one can easily check that for any $\xi\in\mathbb{R}^{n\times p},$
\begin{displaymath}
\|{\rm Proj}_{\mathcal{T}_{X}\mathcal{S}_{n,p}}(\xi)\|_F^2=\|(I_n-XX^\top)\xi\|_F^2+\frac{1}{4}\|X^\top\xi-\xi^\top X\|_F^2,
\end{displaymath}
which together with \Cref{df1} immediately leads to an equivalent description of the first-order stationary point of problem $\cref{ep1}.$
\begin{proposition}\label{xxip}
A point $X\in\mathcal{S}_{n,p}$ is a first-order stationary point of problem $\cref{ep1}$ if and only if there exists $\xi\in\partial F(X)$ such that
\begin{equation}\label{xxi}
(I_n-XX^\top)\xi=0\ \ \text{and}\ \ X^\top\xi-\xi^\top X=0.	
\end{equation}
\end{proposition}
We also need the concept of a retraction.
\begin{definition}(\cite[Definition 3.4]{chen2020proximal})
A retraction on a differential manifold $\mathcal{M}$ is a smooth mapping ${\rm Retr}$ from the tangent bundle $\mathcal{TM}$ onto $\mathcal{M}$ satisfying the following two conditions (here ${\rm Retr}_X$ denotes the restriction of ${\rm Retr}$ onto $\mathcal{T}_{X}\mathcal{M}$):\par
(i) ${\rm Retr}_X(0)=X, \forall X\in\mathcal{M},$ where 0 denotes the zero element of $\mathcal{T}_{X}\mathcal{M}.$\par
(ii) For any $X\in\mathcal{M},$ it holds that
\begin{displaymath}
\lim\limits_{\xi\in\mathcal{T}_{X}\mathcal{M}\to 0}\frac{\|{\rm Retr}_X(\xi)-(X+\xi)\|_F}{\|\xi\|_F}=0.
\end{displaymath}
\end{definition}\par
For the Stiefel manifold, common retractions include the exponential map \cite{edelman1998geometry} and those based on the QR decomposition \cite{absil2009optimization}, Cayley transformation \cite{wen2013feasible}, and polar decomposition \cite{absil2009optimization}. According to \cite{chen2020proximal}, the retraction based on polar decomposition has the lowest cost in terms of computational complexity. Therefore, we will focus on the retraction based on polar decomposition, which is given by
\begin{equation}\label{e39}
	{\rm Retr}_X(\xi)=(X+\xi)(I_p+\xi^\top \xi)^{-\frac{1}{2}}.
\end{equation}\par
\section{Smoothing algorithms for solving problem $\cref{ep1}$}
\label{sec:nu}
In this section, we propose three smoothing algorithms for finding a first-order stationary point of problem $\cref{ep1}.$ In the first subsection, we introduce the Moreau envelope and show that it owns all the nice properties of a smooth approximation function defined in \cite{bian2020smoothing}. In the second subsection, we combine the Moreau envelope with GPC/GRC \cite{gao2018new,wang2020multipliers} and RGD \cite{absil2009optimization} to develop our algorithms for solving problem $\cref{ep1}.$ 
\subsection{Smooth approximation functions and the Moreau envelope}
\label{sec:smoothing}
In \cite[Definition 3.1]{bian2020smoothing}, the authors summarize several important requirements on smoothing functions for a convex function. Based on this, we can formally define smooth approximation functions for a $nonconvex$ function, which slightly modifies the convexity item of the requirements. 
\begin{definition}\label{a2}
The function $\tilde{g}(\cdot,\cdot):\mathbb{R}^{n\times p}\times(0,+\infty)\to\mathbb{R}$ is called a smooth approximation function of a nonsmooth (possibly nonconvex) function $g:\mathbb{R}^{n\times p}\to\mathbb{R}$ if it satisfies the following conditions:\\
(i) $\lim\limits_{(Z,\mu)\to(X,0^+)}\tilde{g}(Z,\mu)=g(X),\ \ \forall X\in\mathbb{R}^{n\times p};$\\
(ii) (convexity) If $g$ is convex, then for any fixed $\mu>0,$ $\tilde{g}(Z,\mu)$ is convex with respect to $Z$ in $\mathbb{R}^{n\times p};$\\  
(iii) (Lipschitz differentiability with respect to $X$) There exists constants $L>0$ and $L^{'}\geq0$ such that for any $\mu\in(0,+\infty),$ $\nabla_{X}\tilde{g}(\cdot,\mu)$ is Lipschitz continuous with Lipschitz constant $L\mu^{-1}+L^{'};$\\
(iv) (Lipschitz continuity with respect to $\mu$) There exists a constant $\kappa>0$ such that 
\begin{displaymath}
|\tilde{g}(X,\mu_1)-\tilde{g}(X,\mu_2)|\leq \kappa|\mu_1-\mu_2|,\ \ \forall \ X\in\mathbb{R}^{n\times p}, \mu_1, \mu_2\in(0,+\infty).
\end{displaymath}
(v) (gradient consistency) $\Big\{\lim\limits_{(Z,\mu)\to(X,0^+)}\nabla_{Z}\tilde{g}(Z,\mu)\Big\}\subseteq\partial g(X),$ $\forall X\in\mathbb{R}^{n\times p}.$ 
\end{definition}\par
Many existing results in \cite{chen1996class,chen2012smoothing,fukushima2002smoothing,hintermuller2013nonconvex,kreimer1992nondifferentiable,nesterov2005smooth,rockafellar2009variational} can help us to construct smoothing functions that satisfy the conditions in \Cref{a2}. In particular, we shall use the Moreau envelope as the smoothing function of $f$ in problem $\cref{ep1}.$ For $\mu>0$ and a convex function $\varphi$ on $\mathbb{R}^{q\times p},$ the Moreau envelope of $\varphi$ with the smoothing parameter $\mu$ at $Z\in\mathbb{R}^{q\times p}$ is defined as
\begin{equation}\label{envm}
	{\rm env}_{\varphi}(Z,\mu):=\min\limits_{Y}\bigg\{\varphi(Y)+\frac{1}{2\mu}\|Y-Z\|^2_F: Y\in\mathbb{R}^{q\times p}\bigg\}.	
\end{equation}
We also define ${\rm env}_{\varphi}(\cdot,0):=\varphi(\cdot).$ The Moreau envelope is closely related to the concept of proximal operator. The proximal operator of a convex function $\varphi$ at $Z\in\mathbb{R}^{q\times p}$ is defined by
\begin{equation}\label{prox}
	{\rm prox}_{\mu \varphi}(Z)=\arg\min\limits_{Y}\bigg\{\varphi(Y)+\frac{1}{2\mu}\|Y-Z\|^2_F: Y\in\mathbb{R}^{q\times p}\bigg\}.
\end{equation}
Next, we verify that the Moreau envelope of a convex and Lipschitz continuous function owns all the properties in \Cref{a2}. For a proper, closed and convex function $\varphi$ on $\mathbb{R}^{q\times p},$ the convexity and smoothness of ${\rm env}_{\varphi}(Z,\mu)$ with $\mu>0$ follow directly from  \cite[Proposition 12.11]{bauschke2011convex} and \cite[Proposition 12.29]{bauschke2011convex} respectively. In particular, the gradient of ${\rm env}_{\varphi}(\cdot,\mu)$ at $Z\in\mathbb{R}^{q\times p}$ given by
\begin{equation}\label{eng}
	\nabla_{Z}{\rm env}_{\varphi}(Z,\mu)=\frac{1}{\mu}(I_q-{\rm prox}_{\mu\varphi})(Z)
\end{equation}
is Lipschitz continuous with Lipschitz constant $\frac{1}{\mu}$ due to the fact that $I_q-{\rm prox}_{\mu\varphi}$ is Lipschitz continuous with Lipschitz constant 1, where $I_q$ stands for the identity operator. Moreover, by \cite[Lemma 2.2]{orabona2012prisma} and further assuming $\varphi$ is Lipschitz continuous with Lipschitz constant $\rho,$ we have that
\begin{displaymath}
{\rm env}_{\varphi}(Z,\mu_1)\leq{\rm env}_{\varphi}(Z,\mu_2)\leq{\rm env}_{\varphi}(Z,\mu_1)+\frac{1}{2}(\mu_1-\mu_2)\rho^2,
\end{displaymath}
for all $\mu_1>\mu_2\geq 0$ and $Z\in\mathbb{R}^{q\times p}.$ This indicates the Lipschitz continuity of ${\rm env}_{\varphi}(\cdot,\mu)$ with respect to $\mu\in[0,\infty).$ Finally, it remains to show that ${\rm env}_{\varphi}(\cdot,\mu)$ satisfies items (i) and (v) of \Cref{a2}, which are established in the following proposition.
\begin{proposition}\label{env}
Let $\varphi:\mathbb{R}^{q\times p}\to\mathbb{R}$ be a convex and Lipschitz continuous function with Lipschitz constant $\rho>0,$ then the following statements hold:\par
(i) $\lim\limits_{Z\to Y,\mu\downarrow 0}{\rm env}_{\varphi}(Z,\mu)=\varphi(Y),\ \forall Y\in\mathbb{R}^{q\times p}.$\par
(ii) $\|\nabla_{Z}{\rm env}_{\varphi}(Z,\mu)\|_F\leq2\rho,\ \forall(Z,\mu)\in\mathbb{R}^{q\times p}\times(0,+\infty).$\par
(iii) $\Big\{\lim\limits_{Z\to Y,\mu\downarrow 0}\nabla_{Z}{\rm env}_{\varphi}(Z,\mu)\Big\}\subseteq\partial \varphi(Y), \forall Y\in\mathbb{R}^{q\times p}.$
\end{proposition} 
\begin{proof}
We first prove Item (i). For any $Y, Z\in\mathbb{R}^{q\times p}$ and $\mu>0,$ and from \cite[Lemma 2.2]{orabona2012prisma}, we have	
\begin{align*}
&\quad\big|{\rm env}_{\varphi}(Z,\mu)-\varphi(Y)\big|\\
&=\big|{\rm env}_{\varphi}(Z,\mu)-\varphi(Z)+\varphi(Z)-\varphi(Y)\big|\\
&\leq \big|{\rm env}_{\varphi}(Z,\mu)-\varphi(Z)\big|+|\varphi(Z)-\varphi(Y)|\\
&\leq \frac{1}{2}\mu\rho^2+\rho\|Z-Y\|_F,
\end{align*}
which yields $\lim\limits_{Z\to Y,\mu\downarrow 0}{\rm env}_{\varphi}(Z,\mu)=\varphi(Y).$\par 
Next we prove Item (ii). By the definition of Moreau envelope \cref{envm} and proximal operator \cref{prox}, for all $Z\in\mathbb{R}^{q\times p}$ and $\mu>0,$ we have that
\begin{equation}\label{envm1}
{\rm env}_{\varphi}(Z,\mu)=\varphi({\rm prox}_{\mu\varphi}(Z))+\frac{1}{2\mu}\|Z-{\rm prox}_{\mu\varphi}(Z)\|_F^2\leq\varphi(Z).
\end{equation}
Since $\varphi$ is Lipschitz continuous with Lipschitz constant $\rho>0,$ we obtain further from \cref{envm1} that
\begin{align*}
\frac{1}{2\mu}\|Z-{\rm prox}_{\mu\varphi}(Z)\|_F^2&\leq \varphi(Z)-\varphi({\rm prox}_{\mu\varphi}(Z))\\
&\leq \rho\|Z-{\rm prox}_{\mu\varphi}(Z)\|_F.
\end{align*}
This together with \cref{eng} indicates that 
\begin{displaymath}
\|\nabla_{Z}{\rm env}_{\varphi}(Z,\mu)\|_F\leq 2\rho,\ \text{for all}\  Z\in\mathbb{R}^{q\times p},\mu>0.
\end{displaymath}
\par
Finally we prove Item (iii). Let $Y_{\ast}\in\Big\{\lim\limits_{Z\to Y,\mu\downarrow 0}\nabla_{Z}{\rm env}_{\varphi}(Z,\mu)\Big\},$ then there exist sequences $\{Z_k\}\to Z_{\ast}$ and $\{\mu_k\}\downarrow 0$ such that sequence $\{Y_k:=\nabla_{Z}{\rm env}_{\varphi}(Z_k,\mu_k)\}$ converges to $Y_{\ast}.$ From Moreau decomposition \cite[Theorem 14.3]{bauschke2011convex}, we have
\begin{equation}\label{ey}
Y_k={\rm prox}_{\varphi^{\ast}/\mu_k}(Z_k/\mu_k),
\end{equation}
where $\varphi^{\ast}$ is the convex conjugate of $\varphi.$ Invoking \cite[Proposition 16.34]{bauschke2011convex}, we know \cref{ey} holds if and only if the following equivalent statement holds:
\begin{equation}\label{vpar}
Z_k-\mu_kY_k\in\partial \varphi^{\ast}(Y_k).	
\end{equation}
Then from \cite[Proposition 2.1.5]{clarke1990optimization}, we obtain by passing to the limit in \cref{vpar} that  $Z_{\ast}\in\partial \varphi^{\ast}(Y_{\ast}).$ According to \cite[Proposition 16.9]{bauschke2011convex}, we can obtain $ Y_{\ast}\in\partial \varphi(Z_{\ast}).$ This completes the proof.	
\end{proof}\par
Let $\tilde{F}(\cdot,\cdot):\mathbb{R}^{n\times p}\times(0,+\infty)\to\mathbb{R}$ be defined by 
\begin{displaymath}
\tilde{F}(X,\mu):={\rm env}_{f}(BX,\mu)+h(X),\ \ \ \forall (X,\mu)\in\mathbb{R}^{n\times p}\times(0,+\infty),
\end{displaymath}
where $f$ and $B$ are given in problem $\cref{ep1}.$ Then we immediately have that
\begin{equation}\label{nfx}
\nabla_{X}\tilde{F}(X,\mu)=B^\top\nabla_{Z}{\rm env}_{f}(Z,\mu)|_{Z=BX}+\nabla h(X).
\end{equation}
Using this and invoking the fact that ${\rm env}_{f}(Z,\mu)$ is a smooth approximation function of $f,$ one can easily check that $\tilde{F}(X,\mu)$ is a smooth approximation function of $F.$ In particular, $\nabla_{X}\tilde{F}(X,\mu)$ is Lipschitz continuous with respect to $X$ and the Lipschitz constant is $L_{\mu}=L_0/\mu+L_h,$ where $L_0=\|B\|^2_2,$ while $\tilde{F}(X,\mu)$ is Lipschitz continuous with respect to $\mu$ with Lipschitz constant $\kappa=L^2_f/2.$ Furthermore, we have the following corollary regarding $\nabla_{X}\tilde{F}(X,\mu).$
\begin{corollary}\label{mf}
Let $S\subseteq\mathbb{R}^{n\times p}$ be a bounded subset. Then it holds that
\begin{displaymath}
\|\nabla_{X}\tilde{F}(X,\mu)\|_F\leq 2L_f\|B\|_2+C_h,\ \forall(X,\mu)\in S\times(0,+\infty),
\end{displaymath}
where $C_h:=\sup\{\|\nabla h(X)\|_F: X\in S\}.$	
\end{corollary}
\begin{proof}
It is a direct consequence of \Cref{env} (ii), \cref{nfx} and the continuity of $\nabla h(X).$
\end{proof}\par
To close this subsection, we point out that for any $Y\in\mathbb{R}^{|\mathcal{E}|\times (N-1)}$ the Moreau envelope of $\Psi$ in \cref{ep} has the form of 
\begin{displaymath}
{\rm env}_{\Psi}(Y,\mu):=\sum_{k=1}^{|\mathcal{E}|}\sum_{j=1}^{N-1}\Phi(Y_{k,j},\mu),
\end{displaymath}
where
\begin{displaymath}
	\Phi(Y_{k,j},\mu)=\begin{cases}
		w_{i_k,j_k}Y_{k,j}-\frac{\mu w_{i_k,j_k}^2}{2},&\mbox{if}\ Y_{k,j}\geq\mu w_{i_k,j_k},\\
		\frac{1}{2\mu}Y_{k,j}^2,&\mbox{if}\ 0\leq Y_{k,j}<\mu w_{i_k,j_k},\\
		0,&\mbox{if}\ Y_{k,j}<0.\end{cases}
\end{displaymath}
\subsection{Numerical algorithms}
\label{sec:mn}
In this subsection, we develop three smoothing algorithms for finding a first-order stationary point of problem $\cref{ep1},$ by combing the smoothing method using Moreau envelope with GPC, GRC and RGD for solving the smooth approximation problem of problem $\cref{ep1}.$ Accordingly, the proposed smoothing algorithms are called smoothing gradient projection with correction $({\rm SGPC})$ algorithm, smoothing gradient reflection with correction $({\rm SGRC})$ algorithm, and smoothing Riemannian gradient descent (${\rm SRGD}$) algorithm, respectively. In what follows, we shall give detailed descriptions of them.\par
\textbf{SGPC/SGRC}\ \ SGPC/SGRC has three steps for each iteration: function value reduction, proximal correction and smoothing parameter updating. The only difference between SGPC and SGRC lies in the function value reduction step. Let $X_k\in\mathcal{S}_{n,p}$ be the current iterate and $\mu_k>0$ be the current smoothing parameter. In the function value reduction step, given a stepsize $\tau_k>0,$ SGPC directly projects $X_k-\tau_k\nabla_{X}\tilde{F}(X_k,\mu_k)$ onto the Stiefel manifold to obtain a feasible intermediate point $\bar{X}_k,$ i.e.,
\begin{equation}\label{e19}
\bar{X}_{k}={\rm Proj}_{\mathcal{S}_{n,p}}(X_k-\tau_k\nabla_{X}\tilde{F}(X_k,\mu_k)).
\end{equation}
For any $Y\in\mathbb{R}^{n\times p},$ the projection ${\rm Proj}_{\mathcal{S}_{n,p}}(Y)$ can be calculated by $RI_pQ^{\top},$ where $R$ and $Q$ are the left and right singular vectors of $Y,$ respectively. In contrast to SGPC, SGRC generates $\bar{X}_k$ by taking the reflection point of $X_k$ on the null space of $U_k:=X_k-\tau_k\nabla_{X}\tilde{F}(X_k,\mu_k),$ that is,
\begin{equation}\label{e11}
	\bar{X}_{k}=(-I_n+2U_k(U_k^\top U_k)^{\dagger}U_k^\top)X_k,
\end{equation}
where $(U_k^\top U_k)^{\dagger}$ denotes the pseudo-inverse of $U_k^\top U_k.$ By \cite[Lemma 3.2 and 3.3]{gao2018new}, the intermediate point $\bar{X}_k$ is guaranteed to have less function value than $X_k$ if the stepsize $\tau_k\in(0,L_{\mu_k}^{-1}).$ Hence, a large $L_{\mu_k}$ yields small stepsize $\tau_k$ and thus may render slow convergence. To remedy this issue, we incorporate a monotone line-search scheme into \cref{e19} and \cref{e11}, which can help to produce an $\bar{X}_k$ with sufficient function value reduction while maintaining an appropriate stepsize $\tau_k.$ Specially, the initial choice of $\tau_k$ is set via the following formula proposed by Barzilai and Browein \cite{barzilai1988two}, which adaptively approximates $L_{\mu_k}^{-1}$ using some local curvature information of the smooth objective, i.e.,
\begin{equation}\label{etau}
	\tau_k=\begin{cases}
		\max\Big\{\underline{\tau}_k,\min\Big\{\bar{\tau}_{k},\frac{\|X_k-X_{k-1}\|^2_F}{\mathbf{T}_k}\Big\}\Big\},&\mbox{if}\ \ \mathbf{T}_k\neq 0,\\
		\bar{\tau}_{k},&\mbox{else},\end{cases}\end{equation}
where $\underline{\tau}_k=\frac{1}{(1+\epsilon)L_{\mu_k}}, \bar{\tau}_k=c\underline{\tau}_k, c>1,$ and
$\mathbf{T}_k:=\langle X_k-X_{k-1}, \nabla_{X}\tilde{F}(X_k,\mu_k)-\nabla_{X}\tilde{F}(X_{k-1},\mu_k)\rangle.$ In the second step, SGPC/SGRC performs a proximal correction to generate a new iterate $X_{k+1}$ as follows
\begin{equation}\label{e26}
	X_{k+1}=\begin{cases}
		\bar{X}_k,&\mbox{if}\ Z_k=0,\\
		-\bar{X}_k{\rm Proj}_{\mathcal{S}_{p,p}}(Z_k),&\mbox{otherwise},
	\end{cases}
\end{equation}
where $Z_k=(\bar{X}_k)^\top\nabla_{X}\tilde{F}(\bar{X}_k,\mu_k)-\gamma I_p$ and $\gamma>0.$ In the final step, SGPC/SGRC adopts a simple criterion from \cite{bian2020smoothing} for updating the smoothing parameter $\mu_k.$ Roughly speaking, if $\tilde{F}(X_{k+1},\mu_k)+\kappa\mu_k$ is sufficiently reduced, then $\mu_k$ does not change, otherwise it will be scaled down properly for the next iteration. We summarize SGPC and SGRC in \Cref{alg1}.\par
\begin{algorithm}[htb]
	\renewcommand{\algorithmicrequire}{\textbf{Input:}}
	\renewcommand{\algorithmicensure}{\textbf{Output:}}
	\caption{SGPC and SGRC algorithms}
	\label{alg1}
	\begin{algorithmic}[1]
		\REQUIRE initial point $X_0\in\mathcal{S}_{n,p}, \mu_{-1}=\mu_0>0,\sigma\in(0,1),c>1,\alpha>0,\epsilon>0,\eta\in(0,1);$ Set $k\leftarrow 0.$
		\WHILE{certain stopping criterion is not reached}
		\STATE 2a) If $k=0,$ choose $\tau_k=1.$ Otherwise compute $\tau_k$ by \cref{etau}.\\
		2b) Compute $\bar{X}_k$ by \cref{e19}(SGPC) or \cref{e11}(SGRC), if $\bar{X}_k$ satisfies
		\begin{equation}\label{ef}
			\tilde{F}(\bar{X}_k,\mu_k)\leq \tilde{F}(X_k,\mu_k)-\frac{\epsilon L_{\mu_k}}{2}\|\bar{X}_k-X_{k}\|^2_F,
		\end{equation}\par
		 then go to 3. Otherwise set $\tau_k=\eta\tau_k,$ and go to 2b).
		\STATE Based on $\bar{X}_k,$ compute $X_{k+1}$ by \cref{e26}.
		\IF {\begin{equation}\label{mu}
				\tilde{F}(X_{k+1},\mu_k)+\kappa\mu_k-\tilde{F}(X_k,\mu_{k-1})-\kappa\mu_{k-1}\leq-\alpha\mu_k^2\end{equation}}
		\STATE $\mu_{k+1}=\mu_k,$
		\ELSE
		\STATE $\mu_{k+1}=\frac{\mu_0}{(k+1)^{\sigma}}.$
		\ENDIF
		\STATE Set $k=k+1$ and return to 2.
		\ENDWHILE
		\RETURN $X_{k}.$
	\end{algorithmic}
\end{algorithm}	
\textbf{SRGD}\ \ Let $X_k\in\mathcal{S}_{n,p}$ and $\mu_k>0$ be the current iterate and smoothing parameter, respectively. Every iteration of SRGD first applies the Riemannian gradient descent to $\tilde{F}(\cdot,\mu_k)$ with backtracking Armijo line-search, and then updates the smoothing parameter $\mu_k.$ According to \cref{proj}, the Riemannian gradient of $\tilde{F}(\cdot,\mu_k)$
at $X_k$ can be directly computed by 
\begin{equation}\label{e45}
	V_k=(I_n-X_kX_k^\top)\nabla_{X}\tilde{F}(X_k,\mu_k)+\frac{1}{2}X_k(X^\top_k\nabla_{X}\tilde{F}(X_k,\mu_k)-\nabla_{X}\tilde{F}(X_k,\mu_k)^\top X_k).
\end{equation}
Since for a stepsize $\tau_k>0,$ $X_k-\tau_k V_k$ does not necessarily lie on $\mathcal{S}_{n,p},$ we perform a retraction to force it back to $\mathcal{S}_{n,p}.$ In particular, a backtracking Armijo line-search procedure is involved to determine the stepsize $\tau_k.$ Furthermore, SRGD updates the parameter $\mu_k$ in the same way as SGPC and SGRC. We summarize it in \Cref{alg2}.\par
\begin{algorithm}[htb]
	\renewcommand{\algorithmicrequire}{\textbf{Input:}}
	\renewcommand{\algorithmicensure}{\textbf{Output:}}
	\caption{SRGD algorithm}
	\label{alg2}
	\begin{algorithmic}[1]
		\REQUIRE initial point $X_0\in\mathcal{S}_{n,p}, \mu_{-1}=\mu_0>0,\sigma\in(0,1),\alpha>0,\eta\in(0,1);$ Set $k\leftarrow 0.$
		\WHILE{certain stopping criterion is not reached}
		\STATE 2a) Based on $X_k,$ compute $V_k$ by \cref{e45}, and set $\tau_k=L_{\mu_k}^{-1}.$\\
		2b) Compute $X_{k+1}={\rm Retr}_{X_k}(-\tau_k V_k),$ if $X_{k+1}$ satisfies
		    \begin{equation}\label{emr}
				\tilde{F}({\rm Retr}_{X_k}(-\tau_k V_k),\mu_k)\leq\tilde{F}(X_k,\mu_{k})-\frac{\tau_k}{2}\|V_k\|_F^2,
			\end{equation}\par
		then go to 3. Otherwise set $\tau_k=\eta\tau_k,$ and go to 2b).
		\IF {\begin{equation}\label{emr1}
			\tilde{F}(X_{k+1},\mu_k)+\kappa\mu_k-\tilde{F}(X_k,\mu_{k-1})-\kappa\mu_{k-1}\leq-\alpha\mu_k^2\end{equation}} 
		\STATE $\mu_{k+1}=\mu_k,$ 
		\ELSE 
		\STATE $\mu_{k+1}=\frac{\mu_0}{(k+1)^{\sigma}}.$
		\ENDIF
		\STATE Set $k=k+1$ and return to 2.
		\ENDWHILE
		\RETURN $X_{k}.$
	\end{algorithmic}
\end{algorithm}

\section{Convergence analysis of the proposed smoothing algorithms}
\label{sec:red}
In this section, we conduct convergence analysis for the proposed algorithms SGPC, SGRC and SRGD. To this end, we first present some basic results which will be used in the convergence analysis of all these smoothing algorithms. Let $\mathcal{A}:=\{k\in\mathbb{N}:\mu_{k+1}\neq\mu_k\},$ 
$m_k:=\sup\{\|\nabla_{X}\tilde{F}(X,\mu_k)\|_F:X\in\mathcal{S}_{n,p}\}$ and $C^{'}_h=\sup\{\|\nabla h(X)\|_F:X\in\mathcal{S}_{n,p}\}.$
\begin{proposition}\label{mu0}
Let the sequences $\{X_k\}$ and $\{\mu_k\}$ be generated by one of the proposed algorithms. Then the following statements hold:\par
(i) $m_k\leq C_1,$ for any $k\in\mathbb{N},$ where
\begin{displaymath}
	C_1:=2L_f\|B\|_2+C^{'}_h.
\end{displaymath} \par	
(ii) there are infinite elements in $\mathcal{A}$ and $\lim\limits_{k\to\infty}\mu_k=0.$	
\end{proposition}
\begin{proof}
(i)	It is a direct consequence of \Cref{mf}.\par
(ii) This item can be proved by following similar arguments to \cite[Lemma 3.2 (ii)]{bian2020smoothing}. For completeness, we give a proof here. Since $\{\mu_k\}$ is nonincreasing, to prove (ii), we suppose that $\lim\limits_{k\to\infty}\mu_k=\widehat{\mu}>0$ by contradiction. Then, $\mu_{k+1}=\frac{\mu_0}{(k+1)^{\sigma}}$ occurs finite times at most, which implies that there exists a constant $K>0$ such that $\mu_k=\widehat{\mu}$ for any $k\geq K.$ Then,
\begin{displaymath}
\widetilde{F}(X_{k+1},\mu_k)+\kappa\mu_k-\widetilde{F}(X_k,\mu_{k-1})-\kappa\mu_{k-1}\leq-\alpha\widehat{\mu}^2,\ \forall k\geq K+1.
\end{displaymath}
We obtain from the above inequality that 
\begin{displaymath}
\lim\limits_{k\to\infty}\widetilde{F}(X_{k+1},\mu_k)+\kappa\mu_k=-\infty.
\end{displaymath}
However, according to $\{X_k\}\subseteq\mathcal{S}_{n,p}$ and the (iv) of \Cref{a2}, this conclusion contradicts
\begin{displaymath}
\widetilde{F}(X_{k+1},\mu_k)+\kappa\mu_k\geq F(X_{k+1})\geq\min\limits_{X\in\mathcal{S}_{n,p}} F(X),\ \forall k\geq K.\end{displaymath}
Thus the proposition is proved.	
\end{proof}\par
Next, we show the convergence of SGPC/SGRC and SRGD in the following two subsections respectively.
\subsection{Convergence analysis of SGPC and SGRC}
\label{mnc}
In \cite{gao2018new,wang2020multipliers}, the authors give convergence analysis of GPC and GRC under the assumption that the objective is twice differentiable and the corresponding Hessian matrix is bounded on a bounded open set containing $\mathcal{S}_{n,p}.$ By a slight modification of the proof, we find that many of their results can be generalized to the scenario where the objective is differentiable with a Lipschitz continuous gradient. Due to this and the Lipschitz differentiability of the objective function $\tilde{F}(X,\mu),$ we are able to apply these results to the function value reduction and proximal correction steps of SGPC and SGRC, which leads to the following lemma and proposition regarding the convergence of SGPC and SGRC.
\begin{lemma}\label{l2}
For any $Z\in\mathcal{B}(X-\tau\nabla_{X}\tilde{F}(X,\mu),\tau \|\nabla_{X}\tilde{F}(X,\mu)\|_F),$ $\tau\in(0,L_{\mu}^{-1}),$ it holds that
\begin{displaymath}
	\tilde{F}(X,\mu)-\tilde{F}(Z,\mu)\geq\frac{1-L_{\mu}\tau}{2\tau}\|X-Z\|^2_F.
\end{displaymath}
\end{lemma}
\begin{proposition}\label{ds}
Let $\{X_k\},\{\bar{X}_k\},\{\mu_k\}$ and $\{\tau_k\}$ be the sequences generated by ${\rm SGPC}$ or ${\rm SGRC}.$ Then for any $k\in\mathbb{N}$ the following statements hold\par
(i) \begin{displaymath}
	\|\bar{X}_k-X_k\|_F\geq \frac{1}{\tau_k^{-1}+m_k}\|(I_n-X_kX_k^\top)\nabla_{X}\tilde{F}(X_k,\mu_k)\|_F.
\end{displaymath}\par
(ii) \begin{displaymath}
	\tilde{F}(\bar{X}_{k},\mu_k)-\tilde{F}(X_{k+1},\mu_k)\geq\frac{\epsilon L_{\mu_k}}{2}\|\bar{X}_{k}-X_{k+1}\|^2_F.
\end{displaymath}\par
(iii) \begin{displaymath}
	\|\bar{X}_k-X_{k+1}\|_F\geq\frac{1}{2(2+\epsilon)L_{\mu_k}}\|X_{k+1}^\top \nabla_{X}\tilde{F}(X_{k+1},\mu_k)-\nabla_{X}\tilde{F}(X_{k+1},\mu_k)^\top X_{k+1}\|_F.
\end{displaymath}
\end{proposition}\par
With the help of \Cref{l2}, we next show that the backtracking line-search procedures of SGPC and SGRC are well-defined.
\begin{proposition}
For any $k\geq 0,$ the backtracking line-search of ${\rm SGPC}$ and ${\rm SGRC}$ terminate at some $\tau_k\geq\underline{\tau}_k\eta$ in at most $\lceil\log\frac{1}{c}/\log\eta\rceil+1$ steps.
\end{proposition}
\begin{proof}
First, it is easy to verify that the $\bar{X}_k$ generated by \cref{e19} or \cref{e11} satisfy $\bar{X}_k\in\mathcal{B}(X_k-\tau_k\nabla_{X}\tilde{F}(X_k,\mu_k),\tau_k \|\nabla_{X}\tilde{F}(X_k,\mu_k)\|_F)$ for any $k\geq 0.$ Invoking \Cref{l2}, we get that for $\tau_k\leq\underline{\tau}_k=\frac{1}{(1+\epsilon)L_{\mu_k}},$
\begin{align*}
	\tilde{F}(X_k,\mu_k)-\tilde{F}(\bar{X}_k,\mu_k)
	&\geq \frac{1-L_{\mu_k}\underline{\tau}_k}{2\underline{\tau}_k}\|X_k-\bar{X}_k\|^2_F\\
	&=\frac{\epsilon L_{\mu_k}}{2}\|X_k-\bar{X}_k\|^2_F,
\end{align*}
which means that \cref{ef} is satisfied. This together with the update scheme of $\tau_k$ indicates that the line-search procedures must terminate at some $\tau_k\geq \underline{\tau}_k\eta.$ By $\bar{\tau}_k=c\underline{\tau}_k$ and a direction calculation, we know that the line-search procedures terminate in at most $\lceil\log\frac{1}{c}/\log\eta\rceil+1$ steps.	
\end{proof}\par
Now we are ready to present the main result of this subsection.
\begin{theorem}\label{t1}
Let $\{X_k\}$ and $\{\mu_k\}$ be the sequences generated by ${\rm SGPC}$ or ${\rm SGRC}.$ Then any accumulation point of $\{X_k: k\in\mathcal{A}\}$ is a first-order stationary point of problem $\cref{ep1}.$
\end{theorem}
\begin{proof}
Since \cref{mu} is violated for $k\in\mathcal{A},$ we have that	
\begin{align}
\alpha\mu_k^2&\geq\tilde{F}(X_k,\mu_{k-1})+\kappa\mu_{k-1}-\tilde{F}(X_{k+1},\mu_k)-\kappa\mu_k\nonumber\\
&\geq \tilde{F}(X_k,\mu_{k})-\tilde{F}(X_{k+1},\mu_{k})\nonumber\\
&=\tilde{F}(X_k,\mu_{k})-\tilde{F}(\bar{X}_k,\mu_{k})+\tilde{F}(\bar{X}_{k},\mu_{k})-\tilde{F}(X_{k+1},\mu_{k}),\label{amu}
\end{align}
where the second inequality follows from the Lipschitz continuity of $\tilde{F}(X_k,\mu)$ with respect to $\mu.$ Now, according to \cref{ef} of SGPC and SGRC and Item (ii) of \Cref{ds}, we obtain further from \cref{amu} that
\begin{displaymath}
\alpha\mu_k^2\geq\frac{\epsilon L_{\mu_k}}{2}(\|\bar{X}_{k}-X_{k+1}\|^2_F+\|\bar{X}_{k}-X_k\|^2_F)
\end{displaymath}
This together with $\lim\limits_{k\to\infty}\mu_k=0$ yields that
\begin{equation}\label{mx}
	\frac{1}{\mu_k}(\|\bar{X}_{k}-X_{k+1}\|_F+\|\bar{X}_{k}-X_k\|_F)\to 0,\  k\in\mathcal{A}\to\infty.
\end{equation}
Hence, it follows immediately that
\begin{equation}\label{xk}
\lim\limits_{k\in\mathcal{A}\to\infty}\frac{1}{\mu_k}\|X_k-X_{k+1}\|_F=0.	
\end{equation}
By the Lipschitz continuity of $\nabla_{X}\tilde{F}(X,\mu_k)$ with respect to $X,$ we get that 
\begin{displaymath}
\|\nabla_{X}\tilde{F}(X_{k+1},\mu_k)-\nabla_{X}\tilde{F}(X_{k},\mu_k)\|_F\leq L_{\mu_k}\|X_{k+1}-X_{k}\|_F.
\end{displaymath}
Combining this and \cref{xk}, we deduce that
\begin{equation}\label{xk1}
\|\nabla_{X}\tilde{F}(X_{k+1},\mu_k)-\nabla_{X}\tilde{F}(X_{k},\mu_k)\|_F\to 0,\ k\in\mathcal{A}\to\infty.	
\end{equation}
On the other hand, invoking the stepsize $\tau_k\geq\underline{\tau}_k,$ we see from Item (i) of \Cref{ds} that
\begin{align*}
\|\bar{X}_{k}-X_{k}\|_F&\geq\frac{1}{(1+\epsilon)L_{\mu_k}+m_k}\|(I_n-X_kX_k^\top)\nabla_{X}\tilde{F}(X_k,\mu_k)\|_F\\	
&=\frac{\mu_k}{(1+\epsilon)L_{\mu_k}\mu_k+m_k\mu_k}\|(I_n-X_kX_k^\top)\nabla_{X}\tilde{F}(X_k,\mu_k)\|_F,\ \forall k\in\mathbb{N}.
\end{align*}
This together with Item (i) of \Cref{mu0} yields that
\begin{equation}\label{xbm}
\frac{1}{\mu_k}\|\bar{X}_{k}-X_{k}\|_F\geq\frac{1}{(1+\epsilon)(L_0+L_h\mu_0)+C_1\mu_0}\|(I_n-X_kX_k^\top)\nabla_{X}\tilde{F}(X_k,\mu_k)\|_F.
\end{equation}
Also, Item (iii) of \Cref{ds} leads to
\begin{align}
&\quad\frac{1}{\mu_k}\|\bar{X}_{k}-X_{k+1}\|_F\label{xb}\\
&\geq\frac{1}{2(2+\epsilon)(L_0+L_h\mu_0)}\|X_{k+1}^\top \nabla_{X}\tilde{F}(X_{k+1},\mu_k)-\nabla_{X}\tilde{F}(X_{k+1},\mu_k)^\top X_{k+1}\|_F.\nonumber
\end{align}
Suppose that $X_{\ast}$ is an accumulation point of $\{X_k: k\in\mathcal{A}\}$ and let $\{X_{k_j}: k_j\in\mathcal{A}\}$ be a subsequence such that $\lim\limits_{j\to\infty}X_{k_j}=X_{\ast}.$ Invoking \cref{xk}, we also know that $\lim\limits_{j\to\infty}X_{k_j+1}=X_{\ast}.$ Since $\{\nabla_{X}\tilde{F}(X_k,\mu_k): k\in\mathcal{A}\}$ is bounded, by passing to a further subsequence if necessary, we may assume without loss of generality that $\lim\limits_{j\to\infty}\nabla_{X}\tilde{F}(X_{k_j},\mu_{k_j})=\xi,$ which belongs to $\partial F(X_{\ast})$ due to Item (v) of \Cref{a2}. Using this and invoking \cref{xk}, we have upon passing to the limit in \cref{xbm} that 
\begin{equation}\label{xa}
(I_n-X_{\ast}X_{\ast}^{\top})\xi=0.	
\end{equation}
Furthermore, $\lim\limits_{j\to\infty}\nabla_{X}\tilde{F}(X_{k_j},\mu_{k_j})=\xi$ implies $\lim\limits_{j\to\infty}\nabla_{X}\tilde{F}(X_{k_j+1},\mu_{k_j+1})=\xi$ due to \cref{xk1}. Combining this and \cref{mx}, we obtain by passing to the limit in \cref{xb} that
\begin{equation}\label{xax}
X_{\ast}^{\top}\xi-\xi^{\top}X_{\ast}=0.
\end{equation}
By \Cref{xxip}, \cref{xa} and \cref{xax}, we finally conclude that $X_{\ast}$ is a first-order stationary point of problem $\cref{ep1}.$ This completes the proof.
\end{proof}
\subsection{Convergence analysis of SRGD}
\label{mrc}
This subsection is devoted to the convergence analysis of SRGD. We begin with some useful properties of the retraction mapping.
\begin{lemma}(\cite{boumal2019global,liu2019quadratic})\label{p7}
Let $\mathcal{M}$ be a compact embedded submanifold of an Euclidean space. There exist constants $M_1>0$ and $M_2>0$ such that for all $X\in\mathcal{M}$ and $\xi\in\mathcal{T}_{X}\mathcal{M},$ the following two inequalities hold:
	\begin{displaymath}
		\|{\rm Retr}_X(\xi)-X\|_F\leq M_1\|\xi\|_F,
	\end{displaymath}
	\begin{displaymath}
		\|{\rm Retr}_X(\xi)-(X+\xi)\|_F\leq M_2\|\xi\|^2_F.
	\end{displaymath}	
\end{lemma}\par
Next, we prove that the backtracking line-search of SRGD must terminate in a finite number of iterations.
\begin{proposition}\label{l10}
For any $k\geq 0,$ the backtracking line-search of ${\rm SRGD}$ terminates at some $\tau_k\geq\tilde{\tau}_k\eta$ in at most $T$ line-search steps, where
\begin{equation}\label{er6}
\tilde{\tau}_k=\frac{1}{M_1^2L_0/\mu_k+2M_2C_1+M_1^2L_h}
\end{equation} 
and 
\begin{equation}\label{er5}
T=\max\bigg\{1,\bigg\lceil \frac{\log L_0-\log (M_1^2L_0+2M_2C_1\mu_0+M_1^2L_h\mu_0)}{\log\eta}\bigg\rceil+1\bigg\}.\end{equation}
\end{proposition}
\begin{proof}
For any $k\geq 0,$ by the fact that $\nabla_{X}\tilde{F}(\cdot,\mu_k)$ is Lipschitz continuous with a constant $L_{\mu_k},$ we have that
\begin{align}
&\quad\tilde{F}({\rm Retr}_{X_k}(-\tau_k V_k),\mu_k)-\tilde{F}(X_{k},\mu_k)\nonumber\\
&\leq \langle \nabla_{X}\tilde{F}(X_k,\mu_k),{\rm Retr}_{X_k}(-\tau_k V_k)-X_k\rangle+\frac{L_{\mu_k}}{2}\|{\rm Retr}_{X_k}(-\tau_k V_k)-X_k\|^2_F\nonumber\\
&=\langle \nabla_{X}\tilde{F}(X_k,\mu_k),{\rm Retr}_{X_k}(-\tau_k V_k)-(X_k-\tau_k V_k)\rangle+\langle \nabla_{X}\tilde{F}(X_k,\mu_k),-\tau_k V_k\rangle\nonumber\\
&\quad+\frac{L_{\mu_k}}{2}\|{\rm Retr}_{X_k}(-\tau_k V_k)-X_k\|^2_F\nonumber\\
&\leq \|\nabla_{X}\tilde{F}(X_k,\mu_k)\|_F\|{\rm Retr}_{X_k}(-\tau_k V_k)-(X_k-\tau_k V_k)\|_F-\tau_k\|V_k\|^2_F\nonumber\\
&\quad+\frac{L_{\mu_k}}{2}\|{\rm Retr}_{X_k}(-\tau_k V_k)-X_k\|^2_F.\label{er3}
\end{align}
Invoking \Cref{mu0} (i), \Cref{p7} and $L_{\mu_k}=L_0/\mu_k+L_h,$ we obtain further from \cref{er3} that
\begin{align}
&\quad\tilde{F}({\rm Retr}_{X_k}(-\tau_k V_k),\mu_k)-\tilde{F}(X_{k},\mu_k)\nonumber\\
&\leq M_2C_1\tau_k^2\|V_k\|^2_F-\tau_k\|V_k\|^2_F+\frac{M_1^2L_{\mu_k}\tau_k^2}{2}\|V_k\|^2_F\nonumber\\
&=\tau_k\bigg(\bigg(\frac{M_1^2L_0}{2\mu_k}+\frac{M_1^2L_h}{2}+M_2C_1\bigg)\tau_k-1\bigg)\|V_k\|^2_F.\label{er4}
\end{align}
Set $\tilde{\tau}_k:=\frac{1}{M_1^2L_0/\mu_k+2M_2C_1+M_1^2L_h}.$ It then follows from $\cref{er4}$ that the backtracking line-search of SRGD terminates at some $\tau_k\geq\tilde{\tau}_k\eta.$\par
For the $k$-th iteration of SRGD, let $T_k\in\mathbb{N}_{+}$ be the number of line-search steps needed to achieve $\tau_k\leq \tilde{\tau}_k.$ Recall that the initial choice of $\tau_k$ is $1/L_{\mu_k},$ then we have that $\eta^{T_k-1}\leq L_{\mu_k}\tilde{\tau}_k.$ One can easily check that $L_{\mu_k}\tilde{\tau}_k\geq \frac{L_0}{M_1^2L_0+2M_2C_1\mu_0+M_1^2L_h\mu_0},$ for all $k\in\mathbb{N}.$ Hence, we conclude that the backtracking line-search steps of SRGD must terminate in at most $T$ line-search steps, where $T$ is given in \cref{er5}.
\end{proof}\par
Now we can establish the convergence of SRGD in the following theorem.
\begin{theorem}\label{t2}
Let $\{X_k\}$ and $\{\mu_k\}$ be the sequences generated by ${\rm SRGD}.$ Then any accumulation point of the sequence $\{X_k: k\in\mathcal{A}\}$ is a first-order stationary point of problem $\cref{ep1}.$
\end{theorem}
\begin{proof}
Since \cref{emr1} is violated for $k\in\mathcal{A},$ by the Lipschitz continuity of $\tilde{F}(X_{k},\mu)$ with respect to $\mu,$ we see that for $k\in\mathcal{A}$
\begin{align}
\alpha\mu_k^2&\geq \tilde{F}(X_k,\mu_{k-1})+\kappa\mu_{k-1}-\tilde{F}(X_{k+1},\mu_k)-\kappa\mu_k\nonumber\\
&\geq \tilde{F}(X_k,\mu_{k})-\tilde{F}(X_{k+1},\mu_{k})\nonumber\\
&\geq \frac{\tilde{\tau}_k\eta}{2}\|V_{k}\|^2_F\nonumber\\
&=\frac{\eta}{2(M_1^2L_0/\mu_k+2M_2C_1+M_1^2L_h)}\|V_{k}\|^2_F,\label{egr}
\end{align}	
where the last inequality follows from \Cref{l10} and $\tilde{\tau}_k$ is defined in \cref{er6}. Using this and invoking $\lim\limits_{k\to\infty}\mu_k=0$, we deduce that
\begin{equation}\label{e53}
\|V_k\|_F\to0,\ \ \ k\in\mathcal{A}\to\infty.
\end{equation}
Now suppose that $X_{\ast}$ is an accumulation point of $\{X_{k}: k\in\mathcal{A}\}$ and let $\{X_{k_j}: k_j\in\mathcal{A}\}$ be a subsequence such that $\lim\limits_{j\to\infty}X_{k_j}=X_{\ast}.$ Since $\{\nabla_{X}\tilde{F}(X_k,\mu_k): k\in\mathcal{A}\}$ is bounded, by passing to a further subsequence if necessary, we may assume without loss of generality that $\lim\limits_{j\to\infty}\nabla_{X}\tilde{F}(X_{k_j},\mu_{k_j})=\xi,$ which belongs to $\partial F(X_{\ast})$ due to Item (v) of \Cref{a2}. Due to this and \cref{e53}, we have upon passing to the limit in \cref{e45} that
\begin{displaymath}
0=(I_n-X_{\ast}X_{\ast}^\top)\xi+\frac{1}{2}X_{\ast}(X^\top_{\ast}\xi-\xi^\top X_{\ast}),
\end{displaymath}
which indicates that
\begin{displaymath}
	0\in{\rm Proj}_{\mathcal{T}_{X_{\ast}}\mathcal{S}_{n,p}}(\partial F(X_{\ast})).
\end{displaymath}
This together with \Cref{df1} implies that $X_{\ast}$ is a first-order stationary point of problem $\cref{ep1}.$ 
\end{proof}	
\section{Numerical experiments}\label{sec:num}
In this section, we apply the proposed smoothing algorithms to the graph Fourier basis problem, which is a special instance of problem $\cref{ep1}.$ We compare the smoothing algorithm with three existing methods: SOC \cite{lai2014splitting}, PAMAL \cite{chen2016augmented} and ManPG-Ada \cite{chen2020proximal}. All the tests are conducted in MATLAB R2021a on a desktop with an Intel Core i7-11800H CPU (2.30GHz) and 16GB of RAM.\par
In our experiment, we first use GSPBOX \cite{perraudin2014gspbox} to generate eight different types of undirected graphs $\mathcal{G}=(\mathcal{V},\mathcal{E}):$ low$\_$stretch$\_$tree, path, comet, random$\_$ring, community, spiral, swiss$\_$roll and sensor. Their number of vertices and edges $(N,|\mathcal{E}|)$ are given in \Cref{tab:4}. Then the matrix $\tilde{B}\in\mathbb{R}^{|\mathcal{E}|\times N}$ in problem \cref{gfb1} is set as \cref{etb}, while $\tilde{V}\in\mathbb{R}^{N\times (N-1)}$ is produced in Matlab by ``null(ones(N,1))".
\begin{table}[htbp]
	\setlength{\abovecaptionskip}{0cm}
	\setlength{\belowcaptionskip}{-0.2cm}
	\small
		\caption{} \label{tab:4}
		\begin{center}
			\begin{tabular}{|c|c|c|c|c|} \hline
				 &\bf low$\_$stretch$\_$tree & \bf path & \bf comet &\bf random$\_$ring \\ \hline
				$(N,|\mathcal{E}|)$ &(4,6) &(8,14) &(12,22) &(16,32)\\ \hline
				  & \bf community &\bf spiral &\bf swiss$\_$roll &\bf sensor \\ \hline
				$(N,|\mathcal{E}|)$ &(22,222) &(35,198) &(110,782) &(120,842)\\ \hline
			\end{tabular}
		\end{center}
	\end{table}{\tiny}\par
The implementation details of the compared methods are given below.\par
\textbf{SOC and PAMAL}\ \ We apply SOC and PAMAL to GFB problem in the same way as those of \cite{sardellitti2017graph}. Both SOC and PAMAL first introduce an auxiliary variable $P\in\mathbb{R}^{N\times(N-1)}$ and rewrite problem \cref{gfb1} as
\begin{equation}\label{epsi}
\begin{cases}
\min\limits_{\tilde{Z},P\in\mathbb{R}^{N\times (N-1)}} \Psi(\tilde{B}\tilde{Z})\\
{\rm s.t.}\ \ \tilde{Z}=P, \tilde{Z}^\top Z_1=0_{N-1}, P^\top P=I_{N-1}.\end{cases}
\end{equation}
Then SOC employs a two block ADMM to solve \cref{epsi}, which updates the iterate as
\begin{displaymath}
	\begin{cases}
		\tilde{Z}_{k+1}=\arg\min\limits_{\tilde{Z}\in\mathbb{R}^{N\times (N-1)}} \{\Psi(\tilde{B}\tilde{Z})+\frac{\beta}{2}\|\tilde{Z}-P_k+\Lambda_k\|_F^2:\tilde{Z}^\top Z_1=0_{N-1}\},\\
		P_{k+1}=\arg\min\limits_{P\in\mathbb{R}^{N\times (N-1)}}\{\|P-(\tilde{Z}_{k+1}+\Lambda_k)\|_F^2:P^\top P=I_{N-1}\},\\
		\Lambda_{k+1}=\Lambda_k+\tilde{Z}_{k+1}-P_{k+1},
	\end{cases}
\end{displaymath}
where $\beta>0$ is a penalty parameter. In contrast to SOC, PAMAL utilizes an inexact augmented Lagrangian method to tackle \cref{epsi}, with the subproblems of $\tilde{Z}$ and $P$ being solved by the proximal alternating minimization (PAM) algorithm \cite{attouch2010proximal}. See \cite[Algorithm 2 and 3]{sardellitti2017graph} for detailed descriptions of PAMAL for \cref{epsi}. Both SOC and PAMAL are terminated when 
\begin{equation}\label{esp}
	\frac{\|\tilde{Z}_k-P_k\|_F}{{\rm max}(1,\|\tilde{Z}_k\|_F+\|P_k\|_F)}<10^{-5},
\end{equation}
which indicates that the violation of constraint $\tilde{Z}=P$ is small enough. We also terminate SOC when its iteration number hits 1000, and terminate PAMAL if its outer iteration number reaches 50. Since the tested graphs are undirected, we make use of their graph Laplacian matrices to generated proper initial points for iteration. Let $D\in\mathbb{R}^{N\times N}$ be a diagonal matrix with $i$-th diagonal entry being $\sum_{j=1}^Nw_{ij}.$ Then the Laplacian matrix is given by $D-W.$ Obviously, $D-W$ has an all-one eigenvector and we set the remaining orthonormal eigenvectors of $D-W$ as the columns of the initial point $\tilde{Z}_0=P_0\in\mathbb{R}^{N\times (N-1)}.$\par
As reported in \cite{chen2020proximal}, SOC and PAMAL are very sensitive to the choice of parameters involved. We make lots of efforts on tuning these parameters and finally adopt the following settings of them in our test. For SOC, we set the penalty parameter $\beta=110.$ For PAMAL, we set $\rho=10, \gamma=1.5, \tau=0.5, c_1=0.5, \Lambda_{{\rm min}}=-10^3, \Lambda_{{\rm max}}=10^3, \Lambda_0=0_{N\times (N-1)}$ and the inner iteration (i.e., using PAM to solve the subproblems of $\tilde{Z}$ and $P$) is terminated when the iteration number exceeds 5. We refer the readers to page 800-802 of \cite{sardellitti2017graph} for the meaning of these parameters in PAMAL.\par
\textbf{ManPG-Ada}\ \ In each iteration of ManPG-Ada for problem \cref{ep2}, the main computational cost lies in solving the following nonsmooth convex optimization problem
\begin{equation}\label{eq3}
V_k:=\arg\min\limits_{V\in\mathcal{T}_{X_k}\mathcal{S}_{N-1,N-1}} \Psi(\tilde{B}\tilde{V}(X_k+V))+\frac{1}{2t}\|V\|_F^2,
\end{equation} 
where $t>0$ is a proximal stepsize. Since the proximal operator of $\Psi\circ (\tilde{B}\tilde{V})$ has no closed-form solution, the semi-smooth Newton method, which is suggested in \cite{chen2020proximal}, can not be efficiently applied to \cref{eq3}. In our tests, we use the convex optimization toolbox CVX \cite{grant2008cvx} to solve \cref{eq3}. We try our best to tune the initial parameter $t$ and finally set $t=100/\|\tilde{B}\tilde{V}\|_2.$ Then in every iteration we time it by 1.01 if no line-search step is needed in the previous iteration, as in implemented in \cite{chen2020proximal}. Moreover, ManPG-Ada is terminated when $\|V_k/t\|_F^2<10^{-8}N^2$ or the maximum iteration number 20 is reached. We also initialize ManPG-Ada at $X_0=\tilde{V}^\top \tilde{Z}_0,$ where $\tilde{Z}_0\in\mathbb{R}^{N\times (N-1)}$ is the initial point for SOC and PAMAL.\par
\textbf{SGPC, SGRC and SRGD}\ \ The proposed algorithms are terminated when
\begin{displaymath}
	\|X_{k+1}-X_{k}\|_F<{\rm tol}_1\ \ \text{and}\ \ \alpha\mu_{k}<{\rm tol}_2,
\end{displaymath}
where ${\rm tol}_1>0$ and ${\rm tol}_2>0$ are two given stopping tolerance. We also terminate them if the iteration number exceeds 10000. Moreover, their inner iterations (i.e., the line-search procedure to find an appropriate stepsize) are terminated when the iteration number hits 50. In particular, we adopt the trick in ManPG-Ada \cite{chen2020proximal} for SRGD to adaptively choose a proper initial stepsize $\tau_k,$ which can improve the performance of SRGD in our experiment. For the three algorithms, we use the same parameters on updating the smoothing parameter $\mu_k:$ $\mu_0=0.1, \eta=0.5, \alpha=10^{-5}|\mathcal{E}|$ and $\sigma=0.8,$ while the stopping tolerances and other parameters are set as in \cref{tab:3}. In addition, this algorithms are all initialized at the same initial point as ManPG-Ada.
\begin{table}[htb]
\setlength{\abovecaptionskip}{0cm}
\setlength{\belowcaptionskip}{-0.2cm}
\small
\renewcommand\arraystretch{1.5}
\begin{center}
\caption{} \label{tab:3}
\begin{tabular}{|c|p{10.5cm}|}
\hline 	SGPC & ${\rm tol}_1=10^{-6}\sqrt{N-1}, {\rm tol}_2=10^{-7}(N-1), \epsilon=10^{-3}$ and $c=10^8.$\\ 
\hline  SGRC & ${\rm tol}_1=2\sqrt{N-1}\times 10^{-5}, {\rm tol}_2=3(N-1)\times 10^{-8}, \epsilon=10^{-3}$ and $c=10^8.$\\ 
\hline  SRGD & ${\rm tol}_1=10^{-6}\sqrt{N-1}, {\rm tol}_2=10^{-8}(N-1)$ and $\tau_0=\mu_0/\|\tilde{B}\tilde{V}\|^2_2.$\\  \hline
\end{tabular}	
\end{center}		
\end{table}\par
The computational results are presented in \Cref{tab:1}. The terms "Time", "Fval", "Orth", "Iter" stand for the CPU time, the function value, $\|X_{k}^\top X_k-I_{N-1}\|_F$ and the number of iterations, respectively. One can observe that the function values found by the proposed smoothing algorithms are very close to those reached by the other compared algorithms. In term of orthogonality, the proposed algorithms are comparable with ManPG-Ada, but much better than SOC and PAMAL, which is not surprising as SOC and PAMAL are both infeasible methods. Moreover, the proposed algorithms substantially outperform the other algorithms in the terms of CPU time. Finally, we also observe that SGPC and SRGD slightly outperform SGRC for most tested graphs.

\begin{center} 
	\begin{longtable}{|c|c|c|c|c|}
		\caption{Computational results on different Graphs} \label{tab:1}\\
		\hline	\endfirsthead
		\hline
		\endhead
		\hline
		\multicolumn{2}{r@{}}{Table 3 continued on the next page}
		\endfoot
		\hline
		\endlastfoot
		\multicolumn{5}{c}{ low$\_$stretch$\_$tree}   \\ \hline
		\bf Algorithm & \bf Time & \bf Fval & \bf Orth & \bf Iter \\ \hline
		SGPC &0.007 &6.000 &6.00e-16 &44 \\
		SGRC &0.012 &6.000 &1.05e-13 &191 \\
		SRGD &0.076 &6.000 &1.96e-15 &754 \\
		SOC  &31.923 &5.998 &3.20e-05 &102 \\
		PAMAL &14.426 &6.000 &2.28e-05 &11 \\
		ManPG-Ada &0.245 &6.000 &8.39e-16 &2 \\ \hline
		\multicolumn{5}{c}{path}   \\ \hline
		SGPC &0.015 &18.020 &3.70e-15 &156 \\
		SGRC &0.169 &18.701 &2.57e-14 &1974 \\
		SRGD &0.036 &18.699 &7.47e-15 &753 \\ 
		SOC  &52.550 &18.698 &7.67e-05 &187 \\
		PAMAL &18.716 &18.698 &3.71e-05 &17 \\
		ManPG-Ada &0.527 &18.699 &4.76e-15 &3 \\ \hline
		\multicolumn{5}{c}{ comet} \\ \hline	
		SGPC &0.074 &23.383 &4.44e-15 &543 \\
		SGRC &0.280 &23.664 &5.00e-14 &2468 \\
		SRGD &0.062 &23.382 &9.68e-15 &809 \\ 
		SOC  &63.944 &23.575 &7.76e-05 &249 \\
		PAMAL &26.907 &23.800 &1.23e-04 &24 \\ 
		ManPG-Ada &0.747 &23.544 &6.39e-15 &4 \\ \hline
		\multicolumn{5}{c}{ random$\_$ring} \\ \hline
		SGPC &0.050 &27.034 &8.48e-15 &218 \\
		SGRC &0.273 &27.031 &4.00e-14 &1768 \\
		SRGD &0.071 &27.029 &1.02e-14 &816 \\ 
		SOC  &23.854 &27.344 &8.51e-05 &93 \\
		PAMAL &24.889 &27.353 &1.38e-04 &22 \\ 
		ManPG-Ada &0.781 &27.032 &1.09e-14 &4 \\\hline
		\multicolumn{5}{c}{ community}   \\ \hline
		SGPC &0.298 &271.317 &1.09e-14 &362 \\
		SGRC &1.519 &271.265 &1.06e-13 &1919 \\
		SRGD &2.747 &271.218 &6.29e-15 &6029 \\ 
		SOC  &351.701 &271.647 &6.47e-05 &954 \\
	    PAMAL &40.528 &265.527 &5.22e-05 &24 \\
		ManPG-Ada &1.272 &271.831 &1.14e-14 &5 \\ \hline
		\multicolumn{5}{c}{spiral}   \\ \hline
		SGPC &0.728 &137.043 &1.49e-14 &391 \\
		SGRC &14.575 &137.041 &2.49e-13 &9919 \\
		SRGD &2.623 &137.015 &1.91e-14 &2862 \\
		SOC  &154.622 &137.204 &1.99e-04 &389 \\
		PAMAL &44.311 &135.786 &1.62e-04 &23 \\
		ManPG-Ada &3.402 &139.717 &1.00e-14 &8 \\ \hline
		\multicolumn{5}{c}{swiss$\_$roll}\\ \hline 
		SGPC &24.547 &245.554 &3.74e-14 &1590 \\
		SGRC &20.898 &247.872 &5.14e-13 &1472 \\
		SRGD &38.416 &244.856 &3.14e-14 &3874 \\
     	SOC  &1892.957 &246.271 &3.60e-04 &287 \\
	    PAMAL &522.412 &247.172 &3.56e-04 &18 \\
		ManPG-Ada &386.042 &247.463 &3.34e-14 &7 \\ \hline
		\multicolumn{5}{c}{sensor}   \\ \hline 
		SGPC &95.230 &272.477 &3.98e-14 &5187 \\
		SGRC &24.864 &276.311 &6.51e-13 &1519 \\
		SRGD &100.246 &272.959 &3.27e-14 &8560 \\
		SOC  &2521.135 &277.056 &3.00e-04 &294 \\
		PAMAL &832.880 &273.979 &4.20e-04 &22 \\
		ManPG-Ada &783.205 &276.167 &3.02e-14 &9 \\ \hline		
	\end{longtable}
\end{center}


\section{Conclusion}
\label{sec:clu}
In this paper, we propose three smoothing algorithms SGPC, SGRC and SRGD for solving a class of nonsmooth and nonconvex optimization problems over the Stiefel manifold (problem $\cref{ep1}$). The proposed algorithms are novel combinations of smoothing methods and GPC, GRC and RGD, which are designed for smooth optimization problems over the Stiefel manifold. It is worth noting that we adopt the Moreau envelope of the nonsmooth convex function $f$ as its smooth approximation function, and prove that the resulting smooth objective function satisfies all the requirements on a smooth approximation function (\Cref{a2}). Using this and a proper scheme for updating the smoothing parameter, we show that any accumulation point of the solution sequence generated by our algorithms is a stationary point of problem $\cref{ep1}.$ As demonstrated in the numerical experiments, the proposed algorithms substantially outperform SOC, PAMAL and ManPG-Ada in the terms of CPU time for solving the graph Fourier basis problem.

\bibliographystyle{siamplain}
	
\end{document}